\documentclass{amsart}
\usepackage[all]{xy}
\usepackage[usenames,dvipsnames]{xcolor}
\usepackage{amsthm}
\usepackage{amssymb}
\usepackage[colorlinks=true]{hyperref}
\usepackage{tikz}
\usetikzlibrary{matrix,arrows}

\numberwithin{equation}{section}

\newtheorem{theorem}[equation]{Theorem} 
\newtheorem*{theorem*}{Theorem}
\newtheorem{lemma}[equation]{Lemma}
\newtheorem{proposition}[equation]{Proposition}
\newtheorem{corollary}[equation]{Corollary}
\newtheorem*{corollary*}{Corollary}
\newtheorem{conjecture}[equation]{Conjecture}

\theoremstyle{remark}

\newtheorem{example}[equation]{Example}

\theoremstyle{remark}
\newtheorem{remark}[equation]{Remark}

\newcommand{\cA}{{\mathcal A}}

\newcommand{\cC}{{\mathcal C}}
\newcommand{\cD}{{\mathcal D}}

\newcommand{\cO}{{\mathcal O}}

\newcommand{\cT}{{\mathcal T}}

\newcommand{\bbC}{\mathbb{C}}

\newcommand{\bbP}{\mathbb{P}}

\newcommand{\bbQ}{\mathbb{Q}}
\newcommand{\bbZ}{\mathbb{Z}}

\DeclareMathOperator{\NChow}{NChow} 
\DeclareMathOperator{\Chow}{Chow} 

\newcommand{\Ab}{\mathrm{Ab}}

\newcommand{\spdgcat}{\mathsf{spdgcat}}
\newcommand{\SmProj}{\mathsf{SmProj}}

\newcommand{\perf}{\mathrm{perf}}

\newcommand{\dg}{\mathsf{dg}}

\newcommand{\op}{\mathsf{op}}

\newcommand{\too}{\longrightarrow}

\newcommand{\ie}{\textsl{i.e.}}

\begin{document}

\title[Chow groups of intersections of quadrics via HPD and NC motives]{Chow groups of intersections of quadrics \\via homological projective duality and \\ (Jacobians
of) noncommutative motives}

\author{Marcello Bernardara and Gon{\c c}alo~Tabuada}

\address{Institut de Math\'ematiques de Toulouse ; UMR5219 - Universit\'e de Toulouse ; CNRS
- UPS IMT, F-31062 Toulouse Cedex 9, France}
\email{marcello.bernardara@math.univ-toulouse.fr} 
\urladdr{http://www.math.univ-toulouse.fr/~mbernard/}

\address{Gon{\c c}alo Tabuada, Department of Mathematics, MIT, Cambridge, MA 02139, USA}
\email{tabuada@math.mit.edu}
\urladdr{http://math.mit.edu/~tabuada/}
\thanks{G.~Tabuada was partially supported by a NSF CAREER Award}
\subjclass[2000]{14A22, 14C15, 14F05, 14J40, 14M10}
\date{\today}

\begin{abstract}
The Beilinson-Bloch type conjectures predict that the low degree rational Chow groups of intersections of quadrics are one dimensional. 
This conjecture was proved by Otwinowska in \cite{otwinowska}. Making use of homological projective duality and the recent theory of 
(Jacobians of) noncommutative motives, we give an alternative proof of this conjecture in the case of a complete intersection of either two quadrics or 
three odd-dimensional quadrics. Moreover, without the use of the powerful Lefschetz theorem, we prove that in these cases the unique non-trivial algebraic 
Jacobian is the middle one. As an application, making use of Vial's work \cite{Vial,vial-fibrations}, we describe the rational Chow motives of these complete
intersections and show that smooth fibrations in such complete intersections over small dimensional bases  $S$ verify Murre's conjecture ($\mathrm{dim}(S)\leq 1$),
Grothendieck's standard conjectures ($\mathrm{dim}(S)\leq 2$), and Hodge's conjecture ($\mathrm{dim}(S)\leq 3$).
\end{abstract}

\maketitle 
\vskip-\baselineskip
\vskip-\baselineskip
\vskip-\baselineskip
\section{Statement of results}
Let $k$ be a field and $X$ a smooth projective $k$-subscheme of $\bbP^n$. Whenever $X$ is a complete intersection of multidegree
$(d_1, \ldots, d_r)$ one has the numerical invariant
$$\kappa := [\frac{n- \sum_{j=2}^r d_j}{d_1}]\,,$$
where $[-]$ denotes the integral part of a rational number. A careful analysis of the different Weil cohomology theories of $X$
led to the following conjecture of Beilinson-Bloch type (explicitly stated by Paranjape in \cite[Conjecture 1.8]{paranjape}):

\vspace{0.2cm}

\begin{conjecture}\label{SPconjecture}
There is an isomorphism $CH_i(X)_\bbQ \simeq \bbQ$ for every $i < \kappa$. 
\end{conjecture}

\vspace{0.2cm}

Otwinowska proved Conjecture \ref{SPconjecture} in the case where $X$ is a complete intersection of quadrics, \ie, when $d_1=\cdots =d_r =2$; see \cite[Cor. 1]{otwinowska}.
If one further assumes that $ \kappa = [\mathrm{dim}(X)/2]$, and that $k \subseteq \bbC$ is algebraically closed, then Conjecture \ref{SPconjecture}
admits an alternative proof. Concretely, our first main result is the following:
%
%
\begin{theorem}\label{thm:main}
Conjecture \ref{SPconjecture} holds when:
\begin{itemize}
\item[(i)] $X$ is a complete intersection of two quadrics;
\item[(ii)] $X$ is a complete intersection of three odd-dimensional quadrics.
\end{itemize}
\end{theorem}

Otwinowska's proof is based on a geometric recursive argument. 
First, one establishes the induction step: if Conjecture \ref{SPconjecture}
holds for complete intersections
of multidegree $(d_1,\ldots,d_r)$, then it also holds for complete intersections of multidegree $(d_1,\ldots,d_r,d_r)$; see \cite[Theorem 1]{otwinowska}.
Then, one uses the fact that Conjecture \ref{SPconjecture} is known in the case of quadric hypersurfaces. One should also mention the work of 
Esnault-Levine-Viehweg \cite{esnault-levine-viehweg}. In {\em loc. cit.}, a geometric proof of Conjecture \ref{SPconjecture}
for very small values of $i$ was obtained via a generalization of Roitman's techniques.
 
Our proof of Theorem~\ref{thm:main} is intrinsically different. It is ``categorical'' in nature and based on recent technology such as Kuznetsov's
homological projective duality
and (Jacobians of) noncommutative motives. In contrast with the proof of Otwinowska, and of Esnault-Levine-Viehweg, it highlights the geometric information contained
in the derived category of $X$ and on its noncommutative motive.
%

Now, recall from Griffiths \cite{Griffiths} the construction of the Jacobians
$J^i(X), 0 \leq i \leq \mathrm{dim}(X)-1$. In contrast with the Picard $J^0(X)=\mathrm{Pic}^0(X)$ and the Albanese $J^{\mathrm{dim}(X)-1}(X)=\mathrm{Alb}(X)$ varieties,
the intermediate Jacobians are not (necessarily) algebraic. Nevertheless, they contain an algebraic torus $J^i_a(X) \subseteq J^i(X)$ defined by the image of
the Abel-Jacobi map
\begin{eqnarray}
AJ^i: A^{i+1}(X)_\bbZ \twoheadrightarrow J^i(X) && 0 \leq i \leq \mathrm{dim}(X)-1\,,
\end{eqnarray}
where $A^{i+1}(X)_\bbZ$ stands for the group of algebraic trivial cycles of codimension $i+1$; consult  Voisin \cite[\S 12]{voisin-book} for further details. In the odd-dimensional cases, \ie, when $\mathrm{dim}(X)=2 \kappa +1$, the following holds:
\begin{itemize}
\item[(i)] When $X$ is a complete intersection of two even-dimensional quadrics, Reid proved in
\cite[Theorem 4.14]{reid:thesis} that $J^{\kappa}_a(X)$ is isomorphic, as a principally polarized abelian variety, to the Jacobian $J(C)$ of the hyperelliptic curve $C$ naturally associated to $X$;
\item[(ii)] When $X$ is a complete intersection of three odd-dimensional quadrics, Beauville proved in  proved in
\cite[Theorem 6.3]{beauvilleprym} that $J^{\kappa}_a(X)$ is isomorphic, as a principally polarized abelian variety, to the Prym variety $\mathrm{Prym}(\widetilde{C}/C)$ naturally associated to $X$.
\end{itemize}
Let $X$ be a complete intersection of either two quadrics of three odd-dimensional quadrics. Making use of Veronese embedding,
we can apply Lefschetz theorem (see \cite[13.2.3]{voisin-book}) to conclude\footnote{The same conclusion follows from Theorem \ref{thm:main} via the injectivity
of the cycle map $CH_i (X)_\bbQ \to H^{2\kappa+1-i}(X,\bbQ)$
and Paranjape's results \cite[Prop. 6.4]{paranjape} (see also \cite[Thm.~1]{otwinowska}).} that $H^{p,q}(X)=0$ if $p\neq q$ and $p+q \neq \kappa$. This implies
that the above Jacobians $J^{\kappa}_a(X)$ are the only non-trivial ones. Our second main result is a proof of this latter fact which avoids the use
of the powerful Lefschetz theorem and Paranjape's results:


\begin{theorem}\label{thm:main2}
\begin{itemize}
\item[(i)] When $X$ is a complete intersection of two odd-dimensional quadrics, we have $J_a^i(X)=0$ for every $i$;
\item[(ii)] When $X$ is a complete intersection of either two even-dimensional quadrics of three odd-dimensional quadrics, we have $J^i_a(X)=0$ for every $i \neq \kappa$.
\end{itemize}
\end{theorem}
 \section{Applications}
%

Our first application, obtained by combining Theorems~\ref{thm:main} and \ref{thm:main2} with Vial's work \cite{Vial} on intermediate algebraic Jacobians, is the following:
\begin{corollary}\label{thm:application1}
Let $X$ be a complete intersection of two odd-dimensional quadrics. In this case we have the following motivic decomposition
\begin{equation}\label{eq:decomp11}
M^i(X)_\bbQ \simeq \left\{ \begin{array}{lcl}
         {\bf L}^{\otimes \frac{i}{2}} & \mbox{if} & 0 \leq i \leq 2d,\,\, i \neq d,\,\, i\, \mbox{even}\\
        ({\bf L}^{\otimes \frac{d}{2}})^{\oplus (d+4)} & \mbox{if}   & i=d\\
        0 & & \mbox{otherwise,}
        \end{array} \right.\,
\end{equation}
where $d:=\mathrm{dim}(X)$ and ${\bf L}$ is the rational Lefschetz motive. 

Let $X$ be a complete intersection of either two even-dimensional quadrics or three odd-dimensional quadrics. In these cases we have the motivic decomposition
\begin{eqnarray}\label{eq:decomp22}
M^i(X)_\bbQ \simeq \left\{ \begin{array}{lcl}
         {\bf L}^{\otimes \frac{i}{2}} & \mbox{if} & 0 \leq i \leq 2d, \,\, i\, \mbox{even}\\
       M^1(J^{\kappa}_a(X))_\bbQ(\kappa-d) & \mbox{if}   & i=d\\
        0 & & \mbox{otherwise.}
        \end{array} \right.\,
\end{eqnarray}
Moreover, in all the above cases the rational Chow motive $M(X)_\bbQ$ is Kimura-finite.
\end{corollary}
Intuitively speaking, Corollary~\ref{thm:application1} is the ``motivic refinement'' of Theorems~\ref{thm:main} and
\ref{thm:main2}. 
Our second application, obtained by combining Theorem~\ref{thm:main} with Vial's work \cite{vial-fibrations} on fibrations, is the following:
\begin{corollary}\label{thm:application2}
Let $f : Y \to B$ be a smooth dominant flat morphism between smooth projective
$k$-schemes. Assume that the fibers of $f$ are either complete intersections of two quadrics or complete intersections
of three odd-dimensional quadrics. Under these assumptions, the following holds:
\begin{itemize}
\item[(i)] When $\mbox{dim}(B) \leq 1$, the rational Chow motive $M(Y)_{\bbQ}$ is Kimura-finite and $Y$ satisfies Murre's conjectures;
\item[(ii)] When $\mbox{dim}(B) \leq 2$, $Y$ satisfies Grothendieck's standard conjectures;
\item[(iii)] When $\mbox{dim}(B) \leq 3$, $Y$ satisfies Hodge's conjecture. 
\end{itemize}
\end{corollary}
%
\subsection*{Notations:} 
Throughout the article we will work over an algebraically closed field $k \subset \bbC$. All $k$-schemes will be smooth and projective.
Given a $k$-scheme $X$, we will write $CH_i(X)_\bbZ$ for the Chow group of $i$-dimensional cycles modulo rational equivalence and $A_i(X)_\bbZ$
for the subgroup of algebraically trivial cycles.  We will also use the codimensional notations $CH^i(X)_\bbZ$ and $A^i(X)_\bbZ$ and rational
coefficients $CH_i(X)_\bbQ, A_i(X)_\bbQ, CH^i(X)_\bbQ, A^i(X)_\bbQ$. Finally, the derived category of perfect complexes of $X$ will be denoted
by $\perf(X)$. Note that the canonical inclusion of categories $\perf(X) \hookrightarrow \cD^b(\mathrm{Coh}(X))$ is an equivalence since $X$ is smooth.
\subsection*{Acknowledgments:} 
The authors are very grateful to H{\'e}l{\`e}ne Esnault for precious advices,
and to Charles Vial and Asher Auel for useful comments and answers. 
\section{Preliminaries}\label{sec:preliminaries}
\subsection*{Differential graded categories}
We will assume that the reader is familiar with the language of differential graded (=dg) categories. The standard reference is Keller's ICM survey \cite{ICM-Keller}. Every (dg) $k$-algebra $A$ gives naturally rise to a dg category $\underline{A}$ with a single object and (dg) $k$-algebra of endomorphisms $A$. Another source of examples is provided by $k$-schemes since, as proved by Lunts-Orlov in \cite[Theorem~2.12]{LO}, the derived category of perfect complexes $\perf(X)$ admits a unique dg enhancement $\perf^\dg(X)$. Recall from Kontsevich \cite{IAS,Miami,finMot} that a dg category $\cA$ is called {\em smooth} if it is perfect as a bimodule over itself and {\em proper} if for each ordered pair of objects $(x,y)$ we have $\sum_i \mathrm{dim}\, H^i \cA(x,y)< \infty$. The main examples are the dg categories $\perf^\dg(X)$.
\subsection*{Noncommutative Chow motives}
Recall from the survey article \cite{survey} the construction of the category $\NChow(k)_\bbQ$ of noncommutative Chow motives (with rational coefficients) and of the $\otimes$-functor
$$ U(-)_\bbQ:\spdgcat(k) \too \NChow(k)_\bbQ\,,$$
where $\spdgcat(k)$ stands for the category of smooth proper dg categories. Given a $k$-scheme $X$, we will write $NC(X)_\bbQ$ instead of $U(\perf^\dg(X))_\bbQ$.
\begin{proposition}\label{prop:semi}
Every semi-orthogonal decomposition $\perf(X)=\langle \cT_1, \ldots, \cT_r\rangle$ gives rise to a direct sum decomposition $NC(X)_\bbQ \simeq U(\cT_1^\dg )_\bbQ \oplus \cdots \oplus U(\cT_r^\dg)_\bbQ$, where $\cT_i^\dg$ stands for the dg enhancement of $\cT_i$ induced from $\perf^\dg(X)$. 
\end{proposition}
\begin{proof}
By construction of the functor $U(-)_\bbQ$, every semi-orthogonal decomposition of $ \perf(X)=\langle \cT,\cT^\perp\rangle$ of length $2$ gives rise to a direct sum decomposition $NC(X)_\bbQ \simeq U(\cT^\dg)_\bbQ \oplus U(\cT^{\perp,\dg})_\bbQ$; consult \cite[Theorem~6.3]{IMRN} for details. The proof follows now from a recursive argument applied to the semi-orthogonal decompositions $\langle \langle \cT_i \rangle, \langle \cT_{i+1}, \ldots, \cT_r \rangle \rangle, 1 \leq i \leq r-1$.
\end{proof}
\section{Jacobians of noncommutative Chow motives}\label{sec:Jacobians}
Recall from Andr{\'e} \cite[\S4]{Andre} the construction of the category $\Chow(k)_\bbQ$ of Chow motives (with rational coefficients) and of the contravariant $\otimes$-functor
$$ M(-)_\bbQ:\SmProj(k)^\op \too \Chow(k)_\bbQ\,,$$
where $\SmProj(k)$ stands for the category of smooth projective $k$-schemes. De Rham cohomology factors through Chow motives. Hence, given an irreducible $k$-scheme $X$ of dimension $d$, one defines
\begin{eqnarray}\label{eq:pairings1}
NH_{dR}^{2i+1}(X) := \sum_{C, \gamma_i} \mathrm{Im}(H^1_{dR}(C) \stackrel{H^1_{dR}(\gamma_i)}{\too} H_{dR}^{2i+1}(X)) && 0 \leq i \leq d-1\,,
\end{eqnarray}
where $C$ is a curve and $\gamma_i:M(C)_\bbQ \to M(X)_\bbQ(i)$ a morphism in $\Chow(k)_\bbQ$. Intuitively speaking, \eqref{eq:pairings1} are the odd pieces of de Rham cohomology that are generated by curves. By restricting the classical intersection bilinear pairings on de Rham cohomology (see \cite[\S3.3]{Andre}) to these pieces one then obtains
\begin{eqnarray}\label{eq:pairings2}
\langle-,-\rangle: NH_{dR}^{2d-2i-1}(X) \times NH_{dR}^{2i+1}(X) \to k && 0 \leq i \leq d-1\,.
\end{eqnarray}
Now, recall from \cite[Theorem~1.3]{MT} the construction of the Jacobian functor
$$ {\bf J}(-):\NChow(k)_\bbQ \too \mathrm{Ab}(k)_\bbQ$$
with values in the category of abelian varieties up to isogeny. As proved in \cite[Theorem~1.7]{MT}, whenever the above pairings \eqref{eq:pairings2} are non-degenerate, one has an isomorphism $ {\bf J}(\perf^\dg(X)) \simeq \prod_{i=0}^{d-1} J^i_a(X)$ in $\mathrm{Ab}(k)_\bbQ$. As explained in {\em loc. cit.}, this is always the case for $i=0$ and $i=d-1$ and all the remaining cases follow from Grothendieck's standard conjecture of Lefschetz type.
\subsection*{Categorical data}\label{sub:categorical}
Let $X$ and $Y$ be two irreducible $k$-schemes of dimensions $d_X$ and $d_Y$, respectively. Assume that they are related by the categorical data:

\vspace{0.1cm}

{\it $(\star)$ There exist semi-orthogonal decompositions $\perf(X) = \langle \cT_X, \cT_X^\perp\rangle$ and $\perf(Y)=\langle \cT_Y, \cT_Y^\perp\rangle$ and an equivalence $\phi: \cT_X \stackrel{\sim}{\to} \cT_Y$ of triangulated categories.}

\vspace{0.1cm}

Let $\Phi$ be the composition $ \perf(X) \to \cT_X \stackrel{\phi}{\simeq} \cT_Y \hookrightarrow \perf(Y)$.
\begin{theorem}{(\cite[Theorem~2.2]{marcell-goncalo-2})}\label{thm:Jac1} 
Assume that the bilinear pairings \eqref{eq:pairings2} (associated to $X$ and $Y$) are non-degenerate and that $\Phi$ is of Fourier-Mukai type.
\begin{itemize}
\item [(i)] Under the above assumptions, one obtains a well-defined morphism
$\tau: \prod_{i=0}^{d_X-1} J^i_a(X) \to \prod_{i=0}^{d_Y-1} J^i_a(Y)$ in $\Ab(k)_\bbQ$. 

\item[(ii)] Assume moreover that ${\bf J}(U(\cT_X^{\perp,\dg})_\bbQ)=0$. This holds for instance whenever $\cT_X^\perp$ admits a full exceptional collection. Under this extra assumption, the morphism $\tau$ is split injective.

\item[(iii)] Assume furthermore that ${\bf J}(U(\cT_Y^{\perp,\dg})_\bbQ)=0$. Under this extra assumption, the morphism $\tau$ becomes an isomorphism.
\end{itemize} 
\end{theorem}
\begin{remark}
Theorem~\ref{thm:Jac1} was proved in {\em loc. cit.} in the case where $k=\bbC$. The same proof works {\em mutatis mutandis} for every algebraically closed field $k \subseteq \bbC$.
\end{remark}
\section{Quadric fibrations and homological projective duality}\label{sec:HPD}
\subsection*{Quadric fibrations}
Let $q:Q \to S$ be a flat quadric fibration of relative dimension $n-1$, $\cC_0$ the sheaf of
even parts of the associated Clifford algebra, and $\perf(S,\cC_0)$ the derived category of perfect complexes over $\cC_0$. 
As proved by Kuznetsov in \cite[Theorem 4.2]{kuznetquadrics}, one has the following semi-orthogonal decomposition
\begin{equation}\label{eq:semiort-for-quad-fib}
\perf(Q) = \langle \perf(S,\cC_0), \perf(S)_1, \ldots, \perf(S)_{n-1} \rangle\,,
\end{equation}
where $\perf(S)_i := q^* \perf(S) \otimes \cO_{Q/S}(i)$. Note that $\perf(S)_i \simeq \perf(S)$
for every $i$. 

In the case where the discriminant $\Delta$ of $Q$ is smooth, the category $\perf(S,\cC_0)$ admits a more geometric description. Concretely, when $n$ is odd (resp. even),
we have a double cover $\widetilde{S} \to S$ ramified along $\Delta$ (resp. a
root stack $\widehat{S}$ with a $\bbZ/2\bbZ$-action along $\Delta$), $\cC_0$ lifts to
an Azumaya algebra with class $\alpha$ in $\mathrm{Br}(\widetilde{S})$ (resp. in 
$\mathrm{Br}(\widehat{S})$), and $\perf(S,\cC_0) \simeq \perf(\widetilde{S},\alpha)$ (resp. $\perf(S,\cC_0) \simeq
\perf(\widehat{S},\alpha)$); consult Kuznetsov \cite[\S 3]{kuznetquadrics} for further details.
%
\subsection*{Intersection of quadrics}
Let $X$ be a smooth complete intersection of $r$ quadric hypersurfaces in $\bbP^n=\bbP(V)$. 
The linear span of these $r$ quadrics  gives rise to a hypersurface $Q \subset \bbP^{r-1} \times \bbP^n$,
and the projection into
the first factor to a flat quadric fibration $q: Q \to \bbP^{r-1}$ of relative dimension $n-1$; 
consult \cite[\S 5]{kuznetquadrics}
for details. As above, we can then consider the sheaf $\cC_0$ of even parts of the associated 
Clifford algebra and the derived
category $\perf(\bbP^{r-1},\cC_0)$ of perfect complexes.

Homological projective duality relates the structure
of $\perf(X)$ with the structure of $\perf(\bbP^{r-1},\cC_0)$. Concretely, 
when $2r < n+1$, one has a fully-faithful
functor $\perf(\bbP^{r-1}, \cC_0) \hookrightarrow \perf(X)$ and the following 
semi-orthogonal decomposition
\begin{equation}\label{eq:semiort-for-c-inters}
\perf(X) = \langle \perf(\bbP^{r-1},\cC_0), \cO_X(1), \ldots, \cO_X(n-2r+1) \rangle\,;
\end{equation}
see \cite[\S 5]{kuznetquadrics}. Here are some (low degree) examples:

\begin{example}({\bf Two odd-dimensional quadrics})\label{ex:2}
Let $r=2$ and $n$ even, so that $q: Q \to \bbP^1$ has odd relative dimension $n-1$.
Thanks to the work of Kuznetsov \cite[Corollary 5.7]{kuznetquadrics}, 
$\perf(\bbP^1,\cC_0)$ is equivalent to the category $\perf(\widehat{\bbP^1})$ 
of the root stack with a $\bbZ/2\bbZ$ stacky structure on the critical points
of $Q \to \bbP^1$. Indeed, in this case $\cC_0$ lifts to a Morita-trivial
Azumaya algebra.
\end{example}
\begin{example}({\bf Two even-dimensional quadrics})\label{subsection:hyp-split-2-quad}
Let $r=2$ and $n$ odd, so that $q: Q \to \bbP^1$ has even relative dimension $n-1$. 
Let us denote by $C$ the hyperelliptic curve naturally
associated to $X$; see \cite[\S4]{reid:thesis}. Similarly to Example \ref{ex:2},
$\perf(\bbP^1,\cC_0)$ is equivalent to the category $\perf(C)$; see \cite[Theorem 2.9]{bondal-orlov}
and \cite[Corollary 5.7]{kuznetquadrics}.
\end{example}

\begin{example}({\bf Three odd-dimensional quadrics})\label{subsection:hyp-split-3-quad}
Let $r=3$ and $n$ even, so that $q:Q\to \bbP^2$ has odd relative dimension $n-1$.
Consider the discriminant divisor of the fibration $q$, which is a curve with at most nodal
singularities; see \cite[Proposition 1.2]{beauvilleprym}.
Let us write $C$ for the normalization of this divisor.
As explained in \cite[Proposition 1.5]{beauvilleprym},
$C$ comes equipped with an \'etale double covering $\widetilde{C}\to C$.
\end{example}

\section{Proof of Theorem~\ref{thm:main} - case ($\mathrm{i}$)}\label{sec:categ-proof}
Let $X$ be a complete intersection of two quadric hypersurfaces in $\bbP^n$ with $n \geq 3$.
The linear span of these two quadrics gives rise to a flat quadric fibration $q:Q \to \bbP^1$
of relative dimension $n-1$. As explained in \S\ref{sec:HPD}, we hence obtain the following
semi-orthogonal decomposition
\begin{eqnarray}\label{eq:decomp2}
\perf(X) & = & \langle \perf(\bbP^1, \cC_0), \cO_X(1), \ldots, \cO_X(n-3)\rangle\,.
\end{eqnarray}
\subsection*{Two odd-dimensional quadrics}
Let $n=2m$ for some integer $m \geq 2$. Since we are intersecting
two odd-dimensional quadrics, we have $\mathrm{dim}(X)=2m-2$ and $\kappa = m-1$. 
Hence, in order to prove Theorem \ref{thm:main}, we need to show that 
\begin{eqnarray}\label{eq:need}
CH_i(X)_\bbQ \simeq \bbQ &\mathrm{for} & i < m-1\,.
\end{eqnarray}
As mentioned in the above Example \ref{ex:2}, one has $\perf(\bbP^1, \cC_0) \simeq
\perf(\widehat{\bbP^1})$ where
$\widehat{\bbP^1}$ is the root stack associated to the quadric fibration $q: Q \to \bbP^1$.
Moreover, as proved by Polishchuk in \cite[Theorem 1.2]{polishchuk}, $\perf(\widehat{\bbP^1})$ admits a full exceptional collection of length $p+2$,
where $p$ stands for the number
of discriminant points. Thanks to Reid \cite[Proposition 2.1]{reid:thesis}, we have $p=n+1=2m+1$. Hence, using Proposition~\ref{prop:semi}
and the above semi-orthogonal decomposition \eqref{eq:decomp2}, we obtain 
\begin{equation}\label{eq:isos}
NC(X)_\bbQ \simeq  U(\perf(\bbP^1, \cC_0)^\dg)_\bbQ \oplus {\bf 1}_\bbQ^{\oplus 2m-3} \simeq  {\bf 1}_\bbQ^{\oplus 2m+3}
\oplus {\bf 1}_\bbQ^{\oplus 2m-3} \simeq  {\bf 1}_\bbQ^{\oplus 4m}\,,
\end{equation}
where ${\bf 1}_\bbQ:=U(\underline{k})$ stands for the $\otimes$-unit of $\NChow(k)_\bbQ$. 
\begin{proposition}\label{prop:2-odd-quadrics-have-finite-chow}
The Chow motive $M(X)_\bbQ$ is of $\bbQ$-Lefschetz type and $CH^\ast(X)_\bbQ$ is a finite $4m$-dimensional $\bbQ$-vector space\footnote{By
construction of the functor $U(-)$, the above isomorphisms \eqref{eq:isos} occur already in the category $\NChow(k)_\bbZ$ of noncommutative
Chow motives with integral coefficients. Hence, using \cite[Theorem~2.1]{marcell-goncalo-1} instead of \cite[Theorems~1.3 and 1.7]{MT1}, we conclude that Proposition \ref{prop:2-odd-quadrics-have-finite-chow} holds also with $\bbQ$ replaced by $\bbZ[(4m-4)!]$.}.
\end{proposition}
\begin{proof}
A noncommutative Chow motive is called of {\em $\bbQ$-unit type} if it is isomorphic in $\NChow(k)_\bbQ$ to
$\oplus_{i=1}^r {\bf 1}_\bbQ$ for some $r >0$. Thanks to the above isomorphisms \eqref{eq:isos},
$NC(X)_\bbQ$ is of $\bbQ$-unit type with $r=4m$. Hence, following \cite[Theorems~1.3 and 1.7]{MT}, the rational Chow motive
$M(X)_\bbQ$ is of $\bbQ$-Lefschetz type. Moreover, there exists a choice of integers $l_1, \ldots, l_{4m} \in
\{ 0, \ldots, 2m-2\}$ giving rise to an isomorphism 
\begin{equation}\label{eq:decomp}
M(X)_\bbQ \simeq {\bf L}^{\otimes l_1} \oplus \cdots \oplus {\bf L}^{\otimes l_{4m}}\,.
\end{equation}
Using \eqref{eq:decomp} we conclude automatically that
$CH^\ast(X)_\bbQ$ is a finite $4m$-dimensional $\bbQ$-vector space. This achieves the proof.
\end{proof}

Thanks to Proposition~\ref{prop:2-odd-quadrics-have-finite-chow}, $M(X)_\bbQ$
is of $\bbQ$-Lefschetz type. 
Consequently, since the identity of $M(X)_\bbQ$ decomposes into a sum of pairwise orthogonal
idempotents of Lefschetz type, the rational cycle class map is an isomorphism
\begin{equation}\label{eq:class}
cl: CH^i(X)_\bbQ \stackrel{\sim}{\too} H^{2i}(X(\bbC), \bbQ)\,.
\end{equation}
%

Now, recall from Reid's thesis \cite[\S 0.7 and Corollary 3.15]{reid:thesis} that
\begin{equation}\label{eq:Reid-computation}
H^{2i}(X(\bbC),\bbQ) \simeq \left\{ \begin{array}{lcl}
         \bbQ & \mbox{if} & i\neq m-1\\
        \bbQ^{\oplus (2m+2)} & \mbox{if}   & i=m-1
        \end{array} \right.\,.
\end{equation}
By combining \eqref{eq:class}\text{-}\eqref{eq:Reid-computation}, we then obtain the following isomorphisms
\begin{equation}\label{eq:decomp-final}
CH_i(X)_\bbQ \simeq CH^{2m-2-i}(X)_{\bbQ} \simeq \left\{ \begin{array}{lcl}
         \bbQ & \mbox{if} & i\neq m-1\\
        \bbQ^{\oplus (2m+2)} & \mbox{if}   & i=m-1
        \end{array} \right.\,.
\end{equation}
This clearly implies \eqref{eq:need} and so the proof is finished.

\subsection*{Two even-dimensional quadrics}
Let $n=2m-1$ for some integer $m \geq 2$. Since we are intersecting two even-dimensional quadrics,
we have $\mathrm{dim}(X)=2m-3$ and $\kappa = m-2$. Hence, in order to prove Theorem \ref{thm:main}, we need to show that 
\begin{eqnarray}\label{eq:need2}
CH_i(X)_\bbQ \simeq \bbQ &\mathrm{for} & i < m-2\,.
\end{eqnarray}
Recall from Example \ref{subsection:hyp-split-2-quad} 
that the hyperelliptic curve naturally associated to $X$ is the discriminant double cover
$C \to \bbP^1$ and that we have an equivalence $\perf(\bbP^1, \cC_0) \simeq \perf(C)$.  
The above semi-orthogonal
decomposition \eqref{eq:decomp2} reduces then to 
\begin{equation}\label{eq:semi-new}
\perf(X)= \langle \perf(C), \cO_X(1), \ldots, \cO_X(2m-4) \rangle\,.
\end{equation}
From \eqref{eq:semi-new} we conclude automatically that $K_0(X)_\bbQ \simeq K_0(C)_\bbQ \oplus \bbQ^{2m-4}$. Using the isomorphisms $K_0(X)_\bbQ \simeq CH^\ast(X)_\bbQ$ and
$K_0(C)_\bbQ \simeq CH^\ast(C)_\bbQ$ we obtain
\begin{equation}\label{eq:Chow-1}
CH^\ast(X)_\bbQ \simeq CH^\ast(C)_\bbQ \oplus \bbQ^{\oplus 2m-4}\,.
\end{equation}
Now, recall from Voisin \cite[Example 21.3.1]{voisin-book}
that the rational Chow ring of every curve $C$ admits the following decomposition $CH^\ast(C)_\bbQ\simeq \bbQ \oplus A_0(C)_\bbQ \oplus \bbQ$.
By combining this decomposition with \eqref{eq:Chow-1}, one hence obtains 
\begin{equation}\label{eq:iso-curve}
CH^\ast(X)_\bbQ \simeq A_0(C)_\bbQ \oplus \bbQ^{\oplus 2m-2}\,.
\end{equation}
\begin{proposition}\label{prop:new}
The inclusion $A_0(C)_\bbQ \subset CH^\ast(X)_\bbQ$ induces an isomorphism $A_0(C)_\bbQ \simeq A_{m-2}(X)_\bbQ$.
\end{proposition}
\begin{proof}
Thanks to the semi-orthogonal decomposition \eqref{eq:semi-new}, the irreducible $k$-schemes $C$ and $X$ are related
by the categorical data $(\star)$ of \S \ref{sub:categorical}. Let us now verify all conditions of
Theorem~\ref{thm:Jac1}. The bilinear pairings \eqref{eq:pairings2} associated to $C$ and $X$ are non-degenerate 
since every
curve and every complete intersection of smooth hypersurfaces in the projective space satisfy Grothendieck's standard
conjecture of Lefschetz type; see Andr{\'e} \cite[\S 5.2.4.5]{Andre} and Grothendieck \cite[\S 3]{grot-standard}. The fact that the inclusion functor $\perf(C)
\hookrightarrow \perf(X)$ is of
Fourier-Mukai type was proved by Bondal-Orlov in \cite[Theorem 2.7]{bondal-orlov}. The remaining conditions
of Theorem~\ref{thm:Jac1} are clearly verified. As a consequence, we obtain an isomorphism
in $\mathrm{Ab}(k)_\bbQ$
\begin{equation}\label{eq:induced-Jac}
\tau:J(C) \stackrel{\sim}{\too} \prod_{i=0}^{2m-4} J_a^i(X) \, = \, J_a^{m-2}(X) \times \prod_{i\neq m-2} J_a^i(X)\,.
\end{equation}
This implies that $J^\kappa_a(X) = J^{m-2}_a(X) \subseteq \tau(J(C))$. As proved by Reid in \cite[Theorem~4.14]{reid:thesis},
we have $J^{m-2}_a(X) \simeq J(C)$ (as principally polarized abelian varieties). As a consequence, we conclude that
$\tau(J(C)) \simeq J^{m-2}_a(X)$. Now, recall once again from Reid \cite[Theorem~4.14]{reid:thesis} that the Abel-Jacobi maps
\begin{eqnarray*}
AJ^1: A^1(C)_\bbZ \twoheadrightarrow J^0(C)=J(C) && AJ^{m-1}: A^{m-1}(X)_\bbZ \twoheadrightarrow J^{m-2}(X)
\end{eqnarray*}
give rise to the isomorphisms $A_0(C)_\bbQ=A^1(C)_\bbQ \simeq J(C)_\bbQ$ and $A_{m-2}(X)_\bbQ = A^{m-1}(X)_\bbQ \simeq J_a^{m-2}(X)_\bbQ$. Making use of them and of $\tau$ we then obtain the desired isomorphism $A_0(C)_\bbQ \simeq A_{m-2}(X)_\bbQ$.
\end{proof}

From Proposition~\ref{prop:new} one obtains the following isomorphism
\begin{equation}\label{eq:quotient1}
CH^\ast(X)_\bbQ/A_0(C)_\bbQ \simeq \bigoplus_{i \neq m-2} CH_i(X)_\bbQ \oplus (CH_{m-2}(X)_\bbQ/A_{m-2}(X)_\bbQ)\,.
\end{equation}
On the other hand, \eqref{eq:iso-curve} gives rise to the isomorphism
\begin{equation}\label{eq:quotient2}
CH^\ast(X)_\bbQ/A_0(C)_\bbQ \simeq \bbQ^{\oplus (2m-2)}\,.
\end{equation}
Now, note that the intersection of the generic $m$-dimensional linear subspace of $\bbP^{2m-1}$ with
$X$ gives rise to a $(m-2)$-dimensional cycle $Z$ on $X$ with $\deg(Z) \neq 0$. This implies that  $Z \notin A_{m-2}(X)_\bbQ$. Since $Z$ is non-torsion, we conclude also that $CH_{m-2}(X)_\bbQ/A_{m-2}(X)_\bbQ \neq 0$. Similarly, by intersecting the 
generic $(i+2)$-dimensional linear subspace of $\bbP^{2m-1}$
with $X$ one obtains a non-torsion $i$-dimensional cycle on $X$. This allows us to conclude that
$CH_i(X)_\bbQ \neq 0$ for all $i \neq m-2$. Finally, a easy dimension counting argument (obtained by combining
\eqref{eq:quotient1}\text{-}\eqref{eq:quotient2} with the fact that $\mathrm{dim}(X)+1 = 2m-2$) implies that $CH_i(X)_\bbQ \simeq \bbQ$ when $i \neq m-2$ and that $CH_{m-2}(X)_\bbQ/A_{m-2}(X)_\bbQ  \simeq \bbQ$. This implies \eqref{eq:need2} and so the proof is finished.

\vspace{0.2cm}
\section{Proof of Theorem~\ref{thm:main} - Case ($\mathrm{ii}$)}
Let $n=2m$ for some integer $m \geq 2$. Since we are intersecting three odd-dimensional quadrics,
we have $\mathrm{dim}(X)=2m-3$ and $\kappa=m-2$. Hence, in order to prove Theorem~\ref{thm:main},
we need to show that 
\begin{eqnarray}\label{eq:need3}
CH_i(X)_\bbQ \simeq \bbQ & \mathrm{for} & i < m-2\,.
\end{eqnarray}
The linear span of these quadrics gives rise to a flat quadric fibration $q:Q \to \bbP^2$
of relative dimension $n-1$,
As explained in
\S\ref{sec:HPD}, one obtains the semi-orthogonal decomposition
\begin{eqnarray}\label{eq:decomp4}
\perf(X) & = & \langle \perf(\bbP^2, \cC_0), \cO_X(1), \ldots, \cO_X(2m-5)\rangle\,.
\end{eqnarray}
From \eqref{eq:decomp4} we conclude automatically that 
\begin{equation}\label{eq:iso-Grothendieck}
K_0(X)_\bbQ \simeq K_0(\perf(\bbP^2, \cC_0))_\bbQ \oplus \bbQ^{\oplus 2m-5}\,.
\end{equation}
Recall from Example \ref{subsection:hyp-split-3-quad} that
$Q \to \bbP^2$ has a discriminant divisor, whose normalization $C$ comes with
an \'etale double cover
$\widetilde{C} \to C$; we write $\iota$ for the involution on $\widetilde{C}$. 
Thanks to the work of Bouali \cite[Corollary 4.3]{bouali}, we have the following isomorphism
\begin{equation}\label{eq:bouali}
CH^*(Q)_\bbQ \simeq 
CH^* (\bbP^2)_\bbQ^{\oplus 2m} \oplus CH_1(\widetilde{C})^-_\bbQ\,,
\end{equation}
where $CH_1(\widetilde{C})^-_\bbQ$ stands for the
$\iota$-anti-invariant part of $CH_1(\widetilde{C})_\bbQ$.
As explained by Beauville in \cite[\S 2]{beauvilleprym}, the \'etale double cover
$\widetilde{C} \to C$ gives rise to a natural splitting of $J(\widetilde{C})=CH_1(\widetilde{C})$
into the $\iota$-invariant part and the $\iota$-anti-invariant part, where the latter identifies with
the Prym variety $\mathrm{Prym}(\widetilde{C}/C)$.
In particular, we have an isomorphism $CH_1(\widetilde{C})^-_\bbQ \simeq A_m(Q)_\bbQ$.
Consequently, since $A_*(\bbP^2)_\bbQ=0$, we conclude that $A_*(Q)_\bbQ =
A_m(Q)_\bbQ \simeq CH_1 (\widetilde{C})^-_\bbQ$.

Now, note that the general semi-orthogonal decomposition \eqref{eq:semiort-for-quad-fib} reduces to
\begin{equation}\label{eq:semi-orthogonal-new}
\perf(Q) = \langle \perf(\bbP^2,\cC_0), \perf(\bbP^2)_1, \ldots \perf(\bbP^2)_{2m-1} \rangle\,.
\end{equation}
Moreover, thanks to Beilinson \cite{beilinson}, $\perf(\bbP^2)$ is generated by $3$ exceptional objects.

\begin{proposition}\label{prop:Jac2}
The Fourier-Mukai functor $\perf(Q) \to \perf(\bbP^2,\cC_0) \hookrightarrow \perf(X)$
induces an isomorphism $A_m(Q)_\bbQ \simeq A_{m-2}(X)_\bbQ$.
\end{proposition}
\begin{proof}
Thanks to the above semi-orthogonal decompositions \eqref{eq:decomp4} and
\eqref{eq:semi-orthogonal-new},
one observes that the irreducible $k$-schemes $Q$
and $X$ are related by the categorical data $(\star)$ of \S \ref{sub:categorical}. Let us now verify
that all conditions of
Theorem~\ref{thm:Jac1} are verified. The bilinear pairings \eqref{eq:pairings2} associated to $Q$ and
$X$ are non-degenerate
since Grothendieck's standard conjecture
of Lefschetz type holds for every complete intersection
of smooth hypersurfaces in the projective space (see Grothendieck \cite[\S 3]{grot-standard})
and for quadric fibrations over projective surfaces
(see Vial \cite[Theorem 7.4]{vial-fibrations}).
The functor $\perf(Q) \to \perf(\bbP^2,\cC_0)
\hookrightarrow \perf(X)$ is of Fourier-Mukai type since it is obtained by homological projective
duality; see \cite[Theorems 5.3 and 5.4]{kuznetquadrics}. The remaining conditions of Theorem~\ref{thm:Jac1}(iii) are clearly verified. As a consequence, one obtains a 
well-defined isomorphism in $\mathrm{Ab}(k)_\bbQ$
\begin{equation}\label{eq:induced-Jac2}
\tau:\prod_{i=0}^{2m} J_a^i(Q) \stackrel{\sim}{\too} \prod_{i=0}^{2m-4} J_a^i(X)\,.
\end{equation}
Thanks to Bouali's computation \eqref{eq:bouali}, we have $J^{i}_a(Q)=0$ when
$i\neq m$. Consequently, the above isomorphism \eqref{eq:induced-Jac2} reduces to
\begin{equation}\label{eq:induced-Jac3}
\tau:J_a^m(Q) \stackrel{\sim}{\too} \prod_{i=0}^{2m-4} J_a^i(X)\,.
\end{equation}
Beauville proved in \cite[Theorems 2.1 and 6.3]{beauvilleprym} that $J^\kappa_a(X)=J^{m-2}_a(X)$ is
isomorphic, as a principally polarized abelian variety, to the Prym variety
$\mathrm{Prym}(\widetilde{C}/C) \simeq J_a^m(Q)$. We hence conclude that $\tau(J_a^m(Q))
= J^{m-2}_a(X)$ and that $J^i_a(X)$ is isogenous to zero when $i\neq m-2$.
Now, recall once again from Beauville 
\cite[Proposition 3.3 and Theorem~6.3]{beauvilleprym} that the Abel-Jacobi maps
\begin{eqnarray*}
AJ^{m}: A^{m+1}(Q)_\bbZ \twoheadrightarrow J^{m}(Q) && AJ^{m-2}: A^{m-1}(X)_\bbZ \twoheadrightarrow J^{m-2}(X)
\end{eqnarray*}
give rise to the isomorphisms $A_m(Q)_\bbQ = A^{m+1}(Q)_\bbQ \simeq J^{m+1}(Q')_\bbQ$
and $A_{m-2}(X)_\bbQ = A^{m-1}(X)_\bbQ \simeq J^{m-2}(X)_\bbQ$. Making use of them and of 
$\tau$ one then obtains the desired isomorphism $A_m(Q)_\bbQ \simeq A_{m-2}(X)_\bbQ$.
\end{proof}

\begin{lemma}\label{lem:prym-in-P2}
There is an isomorphism of $\bbQ$-vector spaces
 \begin{equation}\label{eq:Saito}
 K_0(\perf(\bbP^2,\cC_0))_\bbQ \simeq A_{m}(Q)_\bbQ\oplus \bbQ^{\oplus 3}\,.
 \end{equation}
 \end{lemma}
\begin{proof}
Using the semi-orthogonal decomposition \eqref{eq:semi-orthogonal-new}, we obtain the isomorphism
\begin{equation}\label{eq:K-iso}
K_0(Q)_\bbQ \simeq K_0(\perf(\bbP^2,\cC_0))_\bbQ \oplus \bbQ^{\oplus 6m-3}\,.
\end{equation}
Proposition \ref{prop:Jac2} implies that $A_m(Q)_\bbQ \subset K_0(\perf(\bbP^2,\cC_0))_\bbQ$
via the canonical isomorphism $K_0(Q)_\bbQ \simeq CH^\ast(Q)_\bbQ$. Hence, using the isomorphism
$CH^*(\bbP^2)_\bbQ \simeq \bbQ^3$, we conclude from Bouali's computation
\eqref{eq:bouali} that
\begin{equation}\label{eq:boual2}
CH^*(Q)_\bbQ \simeq A_{m}(Q)_\bbQ \oplus \bbQ^{6m}.
\end{equation}
The desired isomorphism \eqref{eq:Saito} is then obtained by comparing
\eqref{eq:K-iso} with \eqref{eq:boual2} via the
canonical isomorphism $K_0(Q)_\bbQ \simeq CH^*(Q)_\bbQ$.
 \end{proof}

By combining the natural isomorphism $K_0(X)_\bbQ \simeq CH^\ast(X)_\bbQ$ with \eqref{eq:iso-Grothendieck} and \eqref{eq:Saito}, one concludes that
\begin{equation}\label{eq:iso}
CH^\ast(X)_\bbQ \simeq A_m(Q)_\bbQ \oplus \bbQ^{\oplus 2m-2}\,.
\end{equation}
 
From Proposition \ref{prop:Jac2} one obtains the following isomorphism
\begin{equation}\label{eq:quotient3}
CH^\ast(X)_\bbQ/A_m(Q)_\bbQ \simeq \bigoplus_{i \neq m-2} CH_i(X)_\bbQ \oplus \left(CH_{m-2}(X)_\bbQ/A_{m-2}(X)_\bbQ \right)\,.
\end{equation}
On the other hand, \eqref{eq:iso} gives rise to the isomorphism
\begin{equation}\label{eq:quotient4}
CH^\ast(X)_\bbQ/A_m(Q)_\bbQ \simeq \bbQ^{\oplus 2m-2}\,. 
\end{equation}

Now, note that the intersection of the generic $(m+1)$-dimensional linear subspace of $\bbP^{2m-1}$ with
$X$ gives rise to a $(m-2)$-dimensional cycle $Z$ on $X$ with $\deg(Z) \neq 0$. Hence, $Z \notin A_{m-2}(X)_\bbQ$.
Since $Z$ is non-torsion,
we conclude that $CH_{m-2}(X)_\bbQ/A_{m-2}(X)_\bbQ \neq 0$. Similarly, by intersecting the 
generic $(i+3)$-dimensional linear subspace of $\bbP^{2m-1}$
with $X$ one obtains a non-torsion $i$-dimensional cycle on $X$. This allows us to conclude that
$CH_i(X)_\bbQ \neq 0$ for all $i \neq m-2$. Finally, a easy dimension counting argument (obtained by combining \eqref{eq:quotient3}\text{-}\eqref{eq:quotient4} with the fact that $\mathrm{dim}(X)+1=2m-2$) implies that $CH_i(X)_\bbQ \simeq \bbQ$ when $i \neq m-2$ and that $CH_{m-2}(X)_\bbQ/A_{m-2}(X)_\bbQ  \simeq \bbQ$. This implies \eqref{eq:need3} and so the proof is finished.
\section{Proof of Theorem \ref{thm:main2}}
\subsection*{Two odd-dimensional quadrics}
Let $n=2m$ for some integer $m\geq 2$. Since we are intersecting two odd-dimensional quadrics, we have $\mathrm{dim}(X)=2m-2$. Recall from the proof of Proposition~\ref{prop:Jac2} that $X$ satisfies Grothendieck's standard conjecture of Lefschetz type, and from \eqref{eq:isos} that we have the motivic decomposition $NC(X)_\bbQ \simeq {\bf 1}_\bbQ^{\oplus 4m}$. By applying to it the (additive) Jacobian functor ${\bf J}(-)$, we obtain the following isomorphisms in $\mathrm{Ab}(k)_\bbQ$
$$  \prod_{i=0}^{2m-3} J_a^i(X) \simeq {\bf J}(NC(X)_\bbQ) \simeq \bigoplus_{i=1}^{4m} {\bf J}({\bf 1}_\bbQ) \simeq \bigoplus_{i=1}^{4m} {\bf J}(NC(\mathrm{Spec}(k))_\bbQ)\simeq 0\,.$$
This clearly implies that $J_a^i(X)=0$ for every $i$.

\subsection*{Two even-dimensional quadrics}
Let $n=2m-1$ for some integer $m \geq 2$. Since we are intersecting two even-dimensional quadrics, we have $\mathrm{dim}(X)=2m-3$ and $\kappa=m-2$. Recall from \eqref{eq:semi-new} the construction of the semi-orthogonal decomposition
\begin{equation}\label{eq:decomp-aux}
\perf(X)=\langle \perf(C), \cO_X(1), \ldots, \cO_X(2m-4)\rangle
\end{equation}
and from the proof of Proposition \ref{prop:Jac2} that $C$ and $X$ satisfy Grothendieck's standard conjecture of Lefschetz type. Thanks to Proposition~\ref{prop:semi}, \eqref{eq:decomp-aux} gives rise to the motivic decomposition $NC(X)_\bbQ \simeq NC(C)_\bbQ \oplus 
{\bf 1}_\bbQ^{\oplus 2m-4}$. By applying to it the (additive) Jacobian functor ${\bf J}(-)$, we obtain the following isomorphisms in $\mathrm{Ab}(k)_\bbQ$
$$\prod_{i=0}^{2m-4} J_a^i(X) \simeq {\bf J}(NC(X)_\bbQ) \simeq {\bf J}(NC(C)_\bbQ) \oplus {\bf J}(NC(\mathrm{Spec}(k))_\bbQ)^{\oplus 2m-4} \simeq J(C)\,.$$ 
Now, recall from the proof of Proposition \ref{prop:new} that 
$$ J_a^{m-2}(X)_\bbQ \simeq A_{m-2}(X)_\bbQ \simeq A_0(C)_\bbQ \simeq J(C)_\bbQ\,.$$
By combining the above two isomorphisms we conclude that $J_a^i(X)$ is isogenous to zero (and hence is zero) for $i \neq \kappa$.
\subsection*{Three odd-dimensional quadrics}
Let $n=2m$ for some integer $m \geq 2$. Since we are intersecting three odd-dimensional quadrics, 
we have $\mathrm{dim}(X)=2m-3$ and $\kappa=m-2$. We start by recalling a consequence of Conjecture \ref{SPconjecture},
which is based on Bloch-Srinivas methods of decomposition of the diagonal.
\begin{proposition}{(Paranjape \cite[Theorem 6.3]{paranjape})}\label{prop:consequences}
Let $Y$ be a $d$-dimensional complete intersection satisfying Conjecture \ref{SPconjecture}. In this case we have:
\begin{itemize}
\item[(i)] $CH_i(Y)_\bbQ \simeq \bbQ$ for every $i < \kappa$;
\item[(ii)] $CH_i(Y)_\bbQ \simeq \bbQ$ for every $i > d-\kappa$;
\item[(iii)] $CH_{d-\kappa}(Y)_\bbQ$ is a finite dimensional $\bbQ$-vector space.
\end{itemize}
\end{proposition}
\begin{corollary}\label{cor:diagonal}
Let $X$ be a complete intersection of three odd-dimensional quadrics. In this case, $CH_i(X)_\bbQ$ is a
finite dimensional $\bbQ$-vector
space for every $i \neq m-2$.
\end{corollary}
\begin{proof}
Recall from above that $\kappa=m-2$ and $d=2m-3$. Hence, the proof follows from the combination of Theorem \ref{thm:main}
with Proposition \ref{prop:consequences}.
\end{proof}
We now have all the ingredients necessary for the proof.
In codimensional notation Corollary~\ref{cor:diagonal} shows that $CH^i(X)_\bbQ$ is a finite dimensional $\bbQ$-vector space for every $i \neq m-1$. This implies that the subspaces $A^i(X)_\bbQ \subset CH^i(X)_\bbQ, i \neq m-1$, are finite dimensional that that their images under the Abel-Jacobi maps
\begin{eqnarray*}
AJ^{i-1}: A^i(X)_\bbQ \twoheadrightarrow J^{i-1}(X)_\bbQ && i \neq m-1
\end{eqnarray*}
are isogeneous to zero. We conclude then that $J_a^i(X)=0$ for every $i \neq m-2$.
\begin{remark}
The above proof can be easily adapted to the case of a complete intersection of two quadrics.
\end{remark}

\section{Proof of Corollary \ref{thm:application1}}

\subsection*{Two odd-dimensional quadrics}
Let $n=2m$ for some integer $m\geq 2$. Since we are intersecting two odd-dimensional quadrics, we have $d:=\mathrm{dim}(X)=2m-2$. As explained in the proof of
Proposition~\ref{prop:2-odd-quadrics-have-finite-chow}, there exists a choice of integers $l_1, \ldots, l_{4m} \in \{0, \ldots,d\}$ giving rise to an isomorphism
$$ M(X)_\bbQ \simeq {\bf L}^{\otimes l_1} \oplus \cdots \oplus {\bf L}^{\otimes l_{4m}}\,.$$
Note that this isomorphism implies automatically that $M(X)_\bbQ$ is Kimura-finite. Using the following computation (see \eqref{eq:decomp-final})
\begin{equation*}
CH_i(X)_\bbQ \simeq \left\{ \begin{array}{lcl}
         \bbQ & \mbox{if} & i\neq m-1\\
        \bbQ^{\oplus (2m+2)} & \mbox{if}   & i=m-1,
        \end{array} \right.\,
\end{equation*}
one concludes that
\begin{equation}\label{eq:decomp1-1}
M^i(X)_\bbQ = \left\{ \begin{array}{lcl}
         {\bf L}^{\otimes \frac{i}{2}} & \mbox{if} & 0 \leq i \leq 4m-4,\,\, i \neq 2m-2,\,\, i \,\, \mbox{even} \\
        ({\bf L}^{\otimes m-1})^{\oplus 2m+2} & \mbox{if}   & i=2m-2\\
        0 && \mathrm{otherwise.}
        \end{array} \right.\,
\end{equation}
Finally, using the fact that $d=2m-2$, we observe that \eqref{eq:decomp1-1} agrees with \eqref{eq:decomp11}.
\subsection*{Two even-dimensional or three odd-dimensional quadrics}
Let $m$ be an integer $\geq 2$. Since we are intersecting two even-dimensional or three odd-dimensional quadrics,
we have $d:=\mathrm{dim}(X)=2m-3$ and $\kappa=m-2$. Recall that the group $A_i(Y)_\bbQ$ of a $k$-scheme $Y$
is called {\em rationally representable}
if there exists a curve $\Gamma$ and an algebraic surjective map $A_0(\Gamma)_\bbQ \twoheadrightarrow
 A_i(Y)_{\bbQ}$.
\begin{theorem}{(Vial \cite[Theorem~4]{Vial})}\label{thm:Vial2}
Given an irreducible $k$-scheme $Y$ of dimension $d$, the following statements are equivalent:
\begin{itemize}
\item[(i)] One has the following motivic decomposition
\begin{equation}\label{eq:decomp-Vial}
M^j(Y)_\bbQ = \left\{ \begin{array}{lcl}
         ({\bf L}^{\otimes l})^{\oplus a_l} & \mbox{if} & j=2l \\
        M^1(J^l_a(Y))_\bbQ(l-d) & \mbox{if}   & j=2l+1 \,\, 
        \end{array} \right.\,
        \end{equation}
for some integer $a_l \geq 1$.        
\item[(ii)] The groups $A_i(Y_\bbC)_\bbQ, 0 \leq i \leq d$, are rationally representable.
\end{itemize}
\end{theorem}
Now, recall from Theorem~\ref{thm:main2}(ii) that $J^i_a(X)=0$ for $i \neq m-1$. This implies that the groups
$A_i(X)_\bbQ$ are trivial for $i \neq m-2$ and consequently that they are rationally representable. In the case of a complete
intersection of two even-dimensional quadrics,
we have an algebraic isomorphism $A_{m-2}(X)_\bbQ \simeq
A_0(C)_\bbQ \simeq J(C)_\bbQ$; consult Reid \cite[Theorem 4.14]{reid:thesis} for details.
In the case of a complete intersection of three odd-dimensional quadrics, we have an algebraic isomorphism $A_{m-2}(X)_\bbQ \simeq
A_1(Q')_\bbQ \simeq \mathrm{Prym}(\widetilde{C}/C)_\bbQ$ which is induced by an algebraic surjective map $A_0(\widetilde{C}) \twoheadrightarrow
A_{m-2}(X)$;
consult Beauville \cite[Theorems 3.1 and 6.3]{beauvilleprym} for details. Therefore, in both cases, the group $A_{m-2}(X)_\bbQ$ (and hence $A_{m-2}(X_\bbC)_\bbQ$) is also rationally representable. This implies condition (ii) of Theorem~\ref{thm:Vial2}.
As explained in the proof of Theorem~\ref{thm:main}, we have $CH_i(X)_\bbQ \simeq \bbQ$ for $i \neq m-2$ and $CH_{m-2}(X)_\bbQ/A_{m-2}(X)_\bbQ \simeq \bbQ$. As a consequence, all the $a_i$'s are equal to $1$ and the above decomposition
\eqref{eq:decomp-Vial} agrees with \eqref{eq:decomp22}.
\section{Proof of Corollary \ref{thm:application2}}
Let $X$ be a $k$-scheme. Recall that the Chow group $CH_i(X)_\bbQ$ has niveau $\leq r$
 if there exists a closed subscheme $Z \subset X$ of dimension $i + r$ such that the proper pushforward map $CH_i (Z_\bbC)
 \to CH_i (X_\bbC)$ is surjective.
\begin{theorem}{(Vial \cite[Theorem~7.1]{vial-fibrations})}\label{thm:Vial3}
Let $X$ be an irreducible $k$-scheme of dimension $d$. Assume that the Chow groups $CH_0(X_\bbC), \ldots, CH_l(X_\bbC)$
have niveau $\leq n$. Under these assumptions, the following holds:
 \begin{itemize}
 \item[(i)] If $n=3$ and $l=\lfloor \frac{d-4}{2} \rfloor$, then $X$ satisfies Hodge's conjecture;
 \item[(ii)] If $n=2$ and $l=\lfloor \frac{d-3}{2} \rfloor$, then $X$ satisfies Lefschetz's standard conjecture; 
 \item[(iii)] If $n=1$ and $l=\lfloor \frac{d-3}{2} \rfloor$, then $X$ satisfies Murre's conjecture;
 \item[(iv)] If $n=1$ and $l=\lfloor \frac{d-2}{2} \rfloor$, then $M(X)_\bbQ$ is Kimura-finite,
 \end{itemize}
where $\lfloor - \rfloor$ stands for the floor function. 
 \end{theorem}

\begin{theorem}{(Vial \cite[Theorem 6.10]{vial-fibrations})}\label{thm:vial-fibrt}
Let $\pi : X \to B$ be a smooth dominant flat morphism between $k$-schemes. Assume that $CH_i (X_b) = \bbQ$ for all $i < l$ and all closed point $b \in B$. Then, 
$CH_i (X)_\bbQ$ has niveau $\leq \mathrm{dim}(B)$ for all $i < l$.
\end{theorem}
We now have all the ingredients needed for the proof of Corollary~\ref{thm:application2}.
The proof will consist on verifying the conditions of Theorems \ref{thm:Vial3}-\ref{thm:vial-fibrt}. Recall that $f:Y \to B$ is a smooth dominant flat morphism whose fibers are complete intersections either of two quadrics or of three odd-dimensional quadrics. In these cases, Theorem \ref{thm:main}
implies that $CH_i(Y_b)_\bbQ\simeq\bbQ$ for $i < \kappa = [\mathrm{dim}(Y_b)/2]$. As a consequence, $CH_l(Y)$ has niveau $\leq \dim(B)$ for $l \leq [\frac{d-\dim(B)-1}{2}]$. By applying Theorem \ref{thm:Vial3} we then obtain the items
(i)-(iii) of Corollary~\ref{thm:application2}.

\end{document}